\documentclass[11pt,a4paper,final,oneside,openright]{amsart}
\usepackage{amsfonts}
\usepackage{amsmath}
\usepackage{amssymb}
\usepackage[applemac]{inputenc}
\usepackage{url}
\usepackage{hyperref}
\usepackage{color}

\newtheorem{theorem}{{Theorem}}[section]

\newtheorem{isom.ext}[theorem]{{Trivial isometric extension}}

\newtheorem{lemma}[theorem]{{Lemma}}
\newtheorem{corollary}[theorem]{{Corollary}}
\newtheorem{Fact}[theorem]{{Fact}}
\newtheorem{remark}[theorem]{{Remark}}

\newtheorem{question}[theorem]{{Question}}

\definecolor{greenbf}{rgb}{0, 0.7 ,0.3}
 \hypersetup{
colorlinks=true, breaklinks=true, urlcolor= blue, linkcolor= red,
citecolor= {greenbf}, }
\def\C{\mathbb{C}}
\def\R{\mathbb{R}}

\def\U{\sf{U}}
\def\SU{\sf{SU}}

\def\Iso{\sf{Iso}}

\def\SO{{\sf{SO}}}
\def\SU{{\sf{SU}}}

\def\Iso{{\sf{Iso}} }

\begin{document}
\title[Hermite-Lorentz]{On Homogeneous Hermite-Lorentz Spaces of Low
Dimension}
\author{F. Kadi}
\address{Faculty of Exact Sciences, University of El Oued 39000, Algeria.}
\email{kadi-fatma@univ-eloued.dz}
\author{M. Guediri}
\address{Department of Mathematics, College of Sciences, King Saud
University, P.O. Box 2455, Riyadh 11451, Saudi Arabia.}
\email{mguediri@ksu.edu.sa}
\author{A. Zeghib}
\address{UMPA, CNRS, Ecole Normale Superieure de Lyon\\
46, Allee d'Italie 69364 Lyon Cedex 07, France.}
\email{abdelghani.zeghib@ens-lyon.fr }
\subjclass[2000]{53 C 50, 53 C 55.}
\date{\today }
\maketitle
\date{\today }
\begin{abstract} 
We classify irreducible homogeneous almost Hermite-Lorentz spaces of complex
dimension 3, and prove in particular they are geodesically complete.
\end{abstract}

\section{Introduction}

An almost Hermite-Lorentz space $(M, J, g)$ is an almost pseudo-Hermitian
space of signature $- + \ldots +$. So, $J$ is an almost complex structure,
and $g$ is a pseudo-Riemannian metric compatible with $J$ having a signature $- - + \ldots +$. Equivalently,
each tangent space $(T_xM, J_x, g_x)$ is isomorphic to $(\mathbb{C}^n, -
|z_1 |^2 + |z_2|^2 + \ldots + | z_n|^2)$, where $n = \dim_\mathbb{C} M$. We
will use Hermite-Lorentz to mean that $J$ is integrable.

One can think that the complexity of a pseudo-Riemannian metric grows with
signature, and accordingly, its tractability from the algebraic,
geometric, and dynamical point of views decreases with signature. This
explains the special interest of Riemannian and Lorentzian structures whether 
in mathematics or physics. On the other hand, from a complex point of view, Hermite-Lorentz
framework looks like a Lorentzian one in the real case. This motivates the study of Hermite-Lorentz spaces 
and leads to expect some rigidity of them because of the moderate
signature and the cohabitation of the metric with the almost complex structure.

We are interested here in homogeneous almost Hermite-Lorentz spaces. Any Lie
group $L$ of even dimension $2n$ can be endowed with such a structure.
Indeed, identify arbitrarily its Lie algebra $\mathfrak{l}$ to $(\mathbb{C}%
^{n},-|z_{1}|^{2}+|z_{2}|^{2}+\ldots +|z_{n}|^{2})$, and left translate this
linear Hermite-Lorentz structure to all tangent spaces of $L$. The $L$-left
action on itself preserves this structure, and therefore $L$ becomes a
homogeneous almost Hermite-Lorentz space. This convinces us that the
important homogeneous spaces are those having a big isotropy group. (Note that this
is rarely the case for left invariant structures on Lie groups). The symmetric
spaces are of this kind, and they are the most homogeneous spaces after
the constant curvature spaces. In \cite{BZ}, the authors prove that an almost 
Hermite-Lorentz space with big isotropy is necessarily symmetric
(meaning in particular that the almost complex structure is integrable and
the Hermitian metric is K\"ahlerian). Here, the word "big" should be
understood in the sense that the isotropy is $\mathbb{C}$-irreducible. 

The fact that big isotropy implies symmetry is also true in the
(real) Lorentz case, as well as in the quaternionic case (cf. \cite{CM1, CM2}). 
However, this is not true in the Riemannian case and of course in the general pseudo-Riemannian case.

There is also a restriction on the dimension. More precisely, the
(complex) dimension of an almost Hermite-Lorentz space should be $\geq 4$ (cf. \cite{BZ}). 
A similar restriction holds in the quaternionic case (cf. \cite{CM1, CM2}).
It seems to be natural and worthwhile to understand what happens 
for the special dimension $3$.

Let us first observe that the results proved in \cite{BZ} extend well to the 
2-dimensional case. Our purpose here is to investigate the remaining 3-dimensional
case. Roughly speaking, the technical difficulty leading to the restriction to $\dim
\neq 3$ is that, in the real dimension $3$, there exist equivariant cross products $\mathbb{R}^3 \times \mathbb{R}^3 \to \mathbb{R}%
^3$. Our results here show that, in fact, there are many more homogeneous almost
Hermite-Lorentz spaces with $\mathbb{C}$ irreducible isotropy than the symmetric ones. Fortunately, we were able to classify them all.

\subsection{Main Results.}
The following two theorems give a complete description of all 3-dimensional complex homogeneous almost Hermite-Lorentz spaces with $\mathbb{C}$-irreducible isotropy.
 
\begin{theorem}
\label{classification_L} Let $M=G/H$ be a homogeneous almost complex Hermite-Lorentz
 space of complex dimension 3 with an isotropy acting $\mathbb{C}$%
-irreducibly. Then, up to a central cover (i.e. quotienting by a discrete
central subgroup of the universal covering $\tilde{G}$ of $G$), either $M$
is symmetric or $M$ is a 6-dimensional real Lie group $L$ endowed with a
left invariant almost complex Hermite-Lorentz structure and which corresponds to one
of the following four cases:

\begin{enumerate}
\item $L = \mathsf{SL}(2, \mathbb{C}).$

\item $L = \mathsf{SL}(2, \mathbb{R}) \times \mathsf{SL}(2, \mathbb{R}).$

\item $L= \mathsf{SL}(2, \mathbb{R}) \ltimes \mathbb{R}^3$, where $\mathsf{SL%
}(2, \mathbb{R})$ acts on $\mathbb{R}^3$ via its 3-dimensional irreducible
representation (with image $\mathsf{SO}^0(1, 2)$).

\item $L= N_{{\mathfrak{sl}}_2}$, whose Lie algebra ${\mathfrak{n}}_{{%
\mathfrak{sl}}_2}$ is the 2-step 6-dimensional nilpotent algebra with
underlying space ${\mathfrak{sl}}(2, \mathbb{R}) \oplus {\mathfrak{sl}}%
(2, \mathbb{R})$ and bracket $[(u, u^\prime), (v,
v^\prime)] = (0, b(u, v))$, where $b $ is the bracket of the Lie algebra ${%
\mathfrak{sl}}(2, \mathbb{R})$, $b: {\mathfrak{sl}}(2, \mathbb{R}) \times {%
\mathfrak{sl}}(2, \mathbb{R}) \to {\mathfrak{sl}}(2, \mathbb{R})$.
\end{enumerate}

In each of these cases,  and up to automorphism, there is a unique left
invariant almost Hermite-Lorentz structure on $L$ with $\mathbb{C}$%
-irreducible isotropy.

Moreover, the identity component of the isotropy group (of $1 \in L$) in the
automorphism group of the almost Hermite-Lorentz structure is $\mathsf{SO}%
^0(1, 2)$, acting by automorphisms of $L$.

The almost complex structure is integrable only in the
symmetric and the $\mathsf{SL}(2, \mathbb{C})$ cases.
\end{theorem}

\begin{theorem}
\label{classification_G/H} Alternatively, we have the following description
of the left invariant almost Hermite-Lorentz metrics with $\mathbb{C}$%
-irreducible isotropy on 6-dimensional real Lie groups $L$, seen as homogeneous
spaces of the form $G/H$:

\begin{enumerate}
\item For $L = \mathsf{SL}(2, \mathbb{C})$, $G = \mathsf{SL}(2, \mathbb{R})
\times \mathsf{SL}(2, \mathbb{C})$, and $H$ is diagonal in $\mathsf{SL}(2, 
\mathbb{R}) \times \mathsf{SL}(2, \mathbb{R}) \subset \mathsf{SL}(2, \mathbb{%
R}) \times \mathsf{SL}(2, \mathbb{C})$.

\item For $L = \mathsf{SL}(2, \mathbb{R}) \times \mathsf{SL}(2, \mathbb{R}) $%
, $G = \mathsf{SL}(2, \mathbb{R}) \times \mathsf{SL}(2, \mathbb{R}) \times 
\mathsf{SL}(2, \mathbb{R})$ and $H$ is a diagonal embedding of $\mathsf{SL}%
(2, \mathbb{R}).$

\item For $L= \mathsf{SL}(2, \mathbb{R}) \ltimes \mathbb{R}^3$, $G = \mathsf{%
SL}(2, \mathbb{R}) \times (\mathsf{SL}(2, \mathbb{R}) \ltimes \mathbb{R}^3)$
and $H$ is a diagonal embedding of $\mathsf{SL}(2, \mathbb{R}) $ in $\mathsf{%
SL}(2, \mathbb{R}) \times \mathsf{SL}(2, \mathbb{R})$. 
 
\item For $L= N_{{\mathfrak{sl}}_2}$, $G = \mathsf{SL}(2, \mathbb{R})
\ltimes N_{{\mathfrak{sl}}_2}$, $H = \mathsf{SL}(2, \mathbb{R}).$

\end{enumerate}
\end{theorem}

Let us notice here that instead of the semi-direct product $\mathsf{SL}(2, 
\mathbb{R}) \ltimes N_{{\mathfrak{sl}}_2}$ one could consider $\mathsf{SL}%
(2, \mathbb{R}) \ltimes \mathbb{C}^3$, which turns however to be flat.

\subsection{Results in the integrable case.}
With Theorem \ref{classification_L} in hand and  the classification of symmetric
Hermite-Lorentz spaces with $\mathbb{C}$-irreducible isotropy given in \cite{BZ}, we obtain the following corollaries.

\begin{corollary}
\label{integrable_case} Up to a scaling and covering, there are exactly six 3-dimensional complex homogeneous Hermite-Lorentz spaces with $%
\mathbb{C}$-irreducible isotropy:

\begin{enumerate}
\item $\mathsf{SL}(2, \mathbb{C})$ endowed with a left invariant
Hermite-Lorentz metric that is also invariant under the $Ad_{\mathsf{SL}(2, \mathbb{R})}$%
-action.

\item The universal (flat Hermite-Lorentz) complex Minkowski space $\mathsf{%
Mink} _3(\mathbb{C})$ (usually denoted by $\mathbb{C}^{1, 2}$).

\item The complex de Sitter space $\mathsf{dS} _3(\mathbb{C}) = \mathsf{SU}%
(1, 3)/ \mathsf{{U}(1, 2)}.$

\item The complex anti de Sitter space $\mathsf{AdS} _3(\mathbb{C}) = 
\mathsf{SU}(2, 2) / \mathsf{{U}(1, 2)}.$

\item $\mathbb{C} \mathsf{dS} _3 = \mathsf{SO}^0(1, 4) / \mathsf{SO}^0(1, 2)
\times \mathsf{SO}(2).$

\item $\mathbb{C} \mathsf{AdS} _3 = \mathsf{SO}(3, 2)/ \mathsf{SO}(2) \times 
\mathsf{SO}^0(1, 2).$

\bigskip

All of these space are symmetric except for the first one.
\end{enumerate}
\end{corollary}

\begin{corollary}
Let $M$ be an $n$-dimensional complex homogeneous and non-K\"ahlerian
Hermite-Lorentz space with $\mathbb{C}$-irreducible isotropy. Then, $n=3$ and $M$ is
necessarily $\mathsf{SL}(2,\mathbb{C})$ endowed with a left invariant
Hermite-Lorentz metric which is invariant under the action of  $Ad_{\mathsf{SL}(2,%
\mathbb{R})}$.
\end{corollary}

\subsection{Completeness.} Homogeneous pseudo-Riemannian spaces are not automatically complete!  For instance, 
left invariant pseudo-Riemannian metrics on Lie groups are generically thought to be incomplete. In \cite{GL}, it has been shown that incomplete left-invariant Lorentzian metrics on $\mathsf{SL}(2,\mathbb{R})$ exist, and by taking co-compact quotients we obtain locally homogeneous compact Lorentzian manifolds which are incomplete. However, homogeneous compact pseudo-Riemannian manifolds turn out to be complete \cite{Mars}.

We will prove here the following result.

\begin{theorem} \label{completeness_theorem}  Any homogeneous
almost complex Hermite-Lorentz space (of any dimension) with  $\mathbb{C}$%
-irreducible isotropy is geodesically  complete. 
\end{theorem}

This remarkable fact of course let us dare ask the following question. 
\begin{question} {
Is in general any arbitrary homogeneous pseudo-Riemannian space
with irreducible isotropy complete? Perhaps, a hypothesis of weak irreducibility suffices, such as the $\C$-irreducibility here.
}
\end{question}

\subsection{Overview and steps of the proof.}
\label{Overview} Let $M = G/H$ be a $3$-dimensional complex homogeneous almost 
Hermite-Lorentz space with $\mathbb{C}$%
-irreducible isotropy. The first step is to identify $H$ and its infinitesimal
representation in $\mathbb{C}^3$, identified to the tangent space of $G/H$ at a base point.

\subsubsection{The isotropy, the representation ${\mathfrak{hl}}$.}
\label{Isotropy} Consider the irreducible $3$-dimensional real representation
of $\mathsf{SL}(2,\mathbb{R})$, which has $\mathsf{SO}^{0}(1,2)$ as image. It is
more convenient to work at the Lie algebra level, and denote $W_{3}$ to be
this infinitesimal representation of ${\mathfrak{sl}}_{2}(\mathbb{R})$. Let
us note its complexification acting on $\mathbb{C}^{3}=\mathbb{R}^{3}+\sqrt{%
-1}\mathbb{R}^{3}$ by $\mathfrak{hl}$. This will denote the representation ${%
\mathfrak{sl}}(2,\mathbb{R})\rightarrow \mathfrak{gl}(3,\mathbb{C})$
as well as its image, which is nothing but $W_{3}\oplus W_{3}$. It is in fact
the unique $\mathbb{C}$-irreducible representation of ${\mathfrak{sl}}(2,%
\mathbb{R})$ preserving a Hermite-Lorentz product.

\begin{Fact}
{\it 
\label{Isotropy_classification} 
\cite{BZ, SL}. Any sub-algebra ${\mathfrak{h}}\subset \mathfrak{u}(1,2)$
which is $\mathbb{C}$-irreducible must contain a conjugate of ${\mathfrak{hl}%
}$. More precisely, ${\mathfrak{h}}$ equals ${\mathfrak{hl}}$, ${\mathfrak{hl%
}}\oplus {\mathfrak{so}}(2)$, ${\mathfrak{su}}(1,2),$ or ${\mathfrak{u}}(1,2)
$.   }
 \end{Fact}

\subsubsection{The symmetric case.}  \label{symmetric_case}
Observe that for both ${\mathfrak{hl}} \oplus {\mathfrak{so}}(2)$ and ${\mathfrak{u}}%
(1, 2)$, the associated group contains the symmetry $u \to -u $ (of $\mathbb{C}^3$). The space $G/H$
is thus symmetric in this case.

\begin{remark}
Pseudo-Riemannian symmetric spaces have been extensively studied
and partially classified. We refer to the following references \cite{cal}, \cite{hel}, 
\cite{one}, \cite{wolf1}, \cite{wolf2} for generalities on symmetric spaces,
and to \cite{ber} for their classification in the reductive case.
\end{remark}

\subsubsection{Case ${\mathfrak{h}} = {\mathfrak{su}}(1, 2)$.}
We will prove in Section \ref{su(1,2)} that ${\mathfrak{h}} = {\mathfrak{su}}%
(1, 2)$ cannot occur. That is, if the isotropy contains ${\mathfrak{su}}(1,
2)$, then it is in fact equal to ${\mathfrak{u}}(1, 2)$.

\subsubsection{Case ${\mathfrak{h}} = {\mathfrak{hl}} $.}
This is the case that will occupy all the remaining sections. Observe here
that $\dim G = \dim_\mathbb{R} M + \dim {\mathfrak{h}} = 6+3= 9$.

- In Section \ref{semisimple}, we will identify the semi-simple part of $%
G$.

- In Section \ref{Group_G}, we will explicit the isometry group $G$.

- In Section \ref{Group_L}, we will find the normal Lie subgroup $L$ identified
to $G/H$.

- In Section \ref{computations}, we will make the necessary computations to decide in
which case the almost complex structure is integrable and when the
Hermite-Lorentz metric is K\"ahlerian, as well as other curvature descriptions.
This will, in particular, allow us to prove Fact \ref{Full_Group} on the full
isometry group.

- Finally, Section \ref{completeness} will be devoted to the proof of the geodesic completeness of a homogeneous
almost complex Hermite-Lorentz space with  $\mathbb{C}$%
-irreducible isotropy  (Theorem \ref{completeness_theorem}).

\section{Further remarks and constructions}\label{further remarks}

In the present section we will see that most of the spaces that we found in our
classification has become part of a more general setting. This 
gives them more importance, and also raises more natural questions! We hope
returning elsewhere to partially answer some of these questions.

\subsection{Product metric for $L= \mathsf{SL}(2, \mathbb{R}) \times \mathsf{SL}(2, \mathbb{R})$.} \label{product}
In the case $\mathsf{SL}(2, \mathbb{R}) \times \mathsf{SL}(2, 
\mathbb{R})$ the metric is just the product one where each factor is equipped with the
Anti de Sitter structure given by the Killing form. So, as a
pseudo-Riemannian space $\mathsf{SL}(2, 
\mathbb{R}) \times \mathsf{SL}(2, \mathbb{R})$ is symmetric. It is however
not symmetric as an almost Hermite-Lorentz space. The automorphism group of
the last structure is essentially $\mathsf{SL}(2, \mathbb{R})^3$, whereas
the isometry group of the pseudo-Riemannian structure is essentially $%
\mathsf{SL}(2, \mathbb{R})^4$.

\subsection{A hidden direct product in the case $L = \mathsf{SL}(2,\mathbb{R})\ltimes \mathbb{R}^{3}$.} \label{hiden_symmetry}
The construction of the left invariant metric on $\mathsf{SL}(2,\mathbb{R}%
)\ltimes \mathbb{R}^{3}$ can be generalized as follows. Let $A$ be any Lie
group with a Lie algebra ${\mathfrak{a}}$, and consider the semi-direct
product $L=A\ltimes {\mathfrak{a}}$. Let $\kappa $ be the Killing form of ${%
\mathfrak{a}}$, and endow $\mathfrak{l}={\mathfrak{a}} \oplus {\mathfrak{a}}$
with $\kappa \oplus \kappa $. This is always $Ad_{A}$-invariant, and
non-degenerate exactly when $A$ is semi-simple. It is in fact 
pseudo-Hermitian in the general semi-simple case, and of Hermite-Lorentz
type exactly for ${\mathfrak{a}}={\mathfrak{sl}}(2,\mathbb{R})$.

Endowed with this left invariant metric, $A \ltimes {\mathfrak{a}}$ can be seen as a
homogeneous space $M= A \times (A \ltimes {\mathfrak{a}})/ \Delta$ where $%
\Delta$ is the diagonal in $A \times A$. Observe that the direct product $%
L^\prime = A \times {\mathfrak{a}}$ acts freely transitively on $M$. So, our
left invariant metric on $A \ltimes {\mathfrak{a}}$ is isometric to a left
invariant metric on $A \times {\mathfrak{a}}$. This is in fact nothing but
the direct metric product $(A, \kappa) \times ({\mathfrak{a}}, \kappa)$.
Observe that, in this case, the isotropy does not act by automorphisms of $A
\times {\mathfrak{a}}$.

\subsection{2-step nilpotent construction.}
The definition of the 2-step nilpotent Lie algebra ${\mathfrak{n}}_{{%
\mathfrak{sl}}_2}$ is a particular case of the following construction. Let ${%
\mathfrak{a}}$ and ${\mathfrak{b}}$ two linear spaces and $c: {\mathfrak{a}}
\times {\mathfrak{a}} \to {\mathfrak{b}}$ an anti-symmetric bilinear map. On 
${\mathfrak{n}} = {\mathfrak{a}} \times {\mathfrak{b}}$, define a bracket by 
$[u, v] = c(u, v)$, for $u, v \in {\mathfrak{a}}$, and 0 for all other
possibilities.

\subsection{Complexification of semi-simple groups.} \label{complexification}
The construction of the Hermite-Lorentz metric on $\mathsf{SL}(2, \mathbb{C})
$ can be generalized as follows. Let $P$ be a semi-simple real Lie group
with Lie algebra ${\mathfrak{p}}$. Its complexification $P^\mathbb{C}$ has ${\mathfrak{p}}^\mathbb{C} = \mathfrak{p} + \sqrt{-1} {\mathfrak{p}} \sim {%
\mathfrak{p}} \oplus {\mathfrak{p}}$ as a Lie algebra. Endow it with the scalar product $\kappa \oplus \kappa$,
where $\kappa$ is the Killing form of ${\mathfrak{p}}$. This gives rise to a
left invariant pseudo-Hermitian metric on $P^\mathbb{C}$ that is invariant under $%
Ad_P$.

In our case, that is the case ${\mathfrak{p}}^\mathbb{C}=\mathfrak{sl}(2, \mathbb{C})$, we know that $\mathfrak{sl}(2, \mathbb{C})$ is the complexification of each of the two non isomorphic real Lie algebras $\mathfrak{sl}(2, \mathbb{R})$ and $\mathfrak{su}(2)$. It follows that we have two different constructions. The first one corresponds to the case ${\mathfrak{p}}=\mathfrak{sl}(2, \mathbb{R})$, where we obtain a left invariant pseudo-Hermitian metric on $\mathsf{SL}(2, \mathbb{C})$ that is invariant under the adjoint action $Ad_{\mathsf{SL}\left( 2, \mathbb{R}\right)}$. This is in fact a Hermite-Lorentz metric. 

Whereas, in the case ${\mathfrak{p}}=\mathfrak{su}(2)$, we obtain a left invariant pseudo-Hermitian metric on $\mathsf{SL}(2, \mathbb{C})$ that is now invariant under the adjoint action $Ad_{\mathsf{SU}\left( 2\right)}$, and which turns out to be a Riemannian metric (i.e. positive definite).

\subsection{Almost Hermite-Lorentz structures on frame bundles.}
Let $N$ be a 3-Lorentz manifold. The space of its (positively oriented) orthonormal
frames is a principal $\mathsf{SO}^0(1, 2)$-bundle, $P_N \to N$. It turns
out that $P_N$ has a natural almost Hermite-Lorentz structure. Indeed, $TP_N$
splits into a horizontal and a vertical bundles, $TP_N = \mathcal{H }\oplus 
\mathcal{V}$, where $\mathcal{H}$ is given by the connection. Both $\mathcal{%
H}$ and $\mathcal{V}$ are identified to $TM$. This allows us to define the
almost complex structure $J$ sending $\mathcal{H}$ to $\mathcal{V}$. This
also allows us to define Lorentz metrics on $\mathcal{H}$ and $\mathcal{V}$
which makes each isometric to $TM$. We also assume that $\mathcal{H}$ is orthogonal $\mathcal{V}$ in order to
obtain an almost Hermite-Lorentz structure on $P_N$.

The isometry group of $N$ acts freely on $P_N$, and acts transitively
exactly when $N$ is, up to scaling, one of the universal Lorentz spaces of
constant sectional curvature: Minkowski space $\mathsf{Mink} _3$, de Sitter
space $\mathsf{dS} _3 = \mathsf{SO}^0(1, 3)/ \mathsf{SO}^0(1, 2)$, or a
quotient by a central cyclic group of the anti de Sitter space $\widetilde{\mathsf%
{SL}(2, \mathbb{R})}$. So, we have the following identifications:

\begin{itemize}
\item $P_{\mathsf{Mink} _3} = \mathsf{Iso} ^0(\mathsf{Mink} _3) = \mathsf{SO}%
^0(1, 2) \ltimes \mathbb{R}^3$.

\item $P_{\mathsf{dS} _3} = \mathsf{Iso} ^0(\mathsf{dS} _3) = \mathsf{PSL}(2, 
\mathbb{C})$.

\item $P_{\mathsf{PSL}(2, \mathbb{R} )} = \mathsf{Iso} ^0( \mathsf{PSL}(2, 
\mathbb{R}) ) = \mathsf{PSL}(2, \mathbb{R}) \times \mathsf{PSL}(2,\mathbb{R})$.
\end{itemize}

\subsubsection{Integrability.}
A similar construction in the Riemannian case is discussed in  \cite{Ghy} 
(Section 6) where it is proved that the almost complex structure is
integrable exactly if the Riemannian 3-manifold is hyperbolic, i.e. it has
constant negative curvature, say $-1$. In this case, the frame bundle of the
universal cover is identified to $\mathsf{SO}^0(1, 3)$, which is isomorphic
to $\mathsf{PSL}(2, \mathbb{C})$. We think that a similar characterization is
also true in our (Lorentz) case: the almost complex structure is integrable exactly
when $N$ has, up to scaling, a constant curvature equal to $-1$, and so it is locally
isometric to $\mathsf{dS} _3$, in which case the frame bundle has the
complex structure of $\mathsf{PSL}(2, \mathbb{C})$.

\subsubsection{On the tangent bundle.}
The tangent bundle $TN$, too, has a natural almost Hermite-Lorentz
structure. The isometry group $\mathsf{Iso} (N)$ does not act transitively in
this case, since the subsets $T^rN$ of vectors of a given square length $r$ 
are preserved. In the constant curvature case, the isometry group acts
transitively on these subsets. But $TN$  may have extra-isometries which do not
come from $\mathsf{Iso} (N)$. For example, $T \mathsf{Mink} _3$ is $%
\mathbb{C}^3$ with its usual  Hermite-Lorentz structure (and hence
homogeneous).

It turns out that the K\"ahler form of $TN$ is nothing but the push-forward by
the dual isomorphism $TN \to T^*N$ of the  canonical symplectic form of  
$T^*N$. Therefore, $TN$ is always almost K\"ahlerian \cite{TO}.

\subsection{Compact models.}
When we have a geometry (in Klein sense) $(G,G/H)$, we are interested to see if
it admits a compact model. Say, in our four cases where $G/H$ can be identified
to a Lie group $L$ one can ask if $L$ has a compact quotient that locally has the same almost Hermite-Lorentz structure? Formally, one can ask if $L$
has a co-compact lattice $\Gamma $?

It turns out that this is the case for $\mathsf%
{SL}(2,\mathbb{R})\times \mathsf{SL}(2,\mathbb{R})$ and $\mathsf{SL}(2,%
\mathbb{C})$, because of classical constructions, or thanks to Borel's
Theorem saying that semi-simple Lie groups admit co-compact lattices (see
for instance \cite{Zim}). It is also the case for $N_{{\mathfrak{sl}}_{2}}$,
because it is a rational nilpotent Lie group (see for instance  \cite{Rag},
Theorem 2.12). As for $\mathsf{SL}(2,\mathbb{R})\ltimes \mathbb{R}^{3}$, it
turns out that one can classify its lattices. They are, up to finite index, isomorphic to $%
\mathsf{SL}(2,\mathbb{Z})\ltimes \mathbb{Z}^{3}$, and so they are not
co-compact. Fortunately, as we have observed above, this space has also the product
presentation $\mathsf{SL}(2,\mathbb{R})\times \mathbb{R}^{3}$ and so admits
co-compact lattices.

The left isometric action of $\mathsf{SL}(2, \mathbb{R})$ on $L$ passes to
the quotients $L / \Gamma$. This  is an interesting isometric action with
strong dynamics. For instance, it is ergodic by Moore's Theorem (see for
instance \cite{Zim}).  

Observe that these $\Gamma$'s can be deformed in the full isometry group $L
\times \mathsf{SL}(2, \mathbb{R})$ and one then obtains "non-standard''
examples (compare for instance with \cite{Ghy}).

\subsubsection{The symmetric case.}
The symmetric cases listed in Corollary \ref{integrable_case} do not admit
a $L$-representation'', i.e. they are not identified to left invariant
almost Hermite-Lorentz structures on Lie groups. It is worthwhile to
investigate the existence (and probably the non-existence) of compact models in each of
these cases.

\section{Case of isotropy ${\mathfrak{h}} = {\mathfrak{su}}(1, 2)$.}\label{su(1,2)}
Here, we want to prove that if the isotropy  algebra contains ${\mathfrak{su}}(1,2)$,
then the space is K\"ahlerian  with
a constant holomorphic  sectional curvature, and  in fact its isotropy algebra is rather ${\mathfrak{u}}(1,2)$.

A Lie-theoretic proof seems to be difficult, since $\dim G$ would be equal to 14 and there will be many cases to consider. Our proof here is rather geometrical. 
So, let $(M=G/H, J, g) $ be an almost Hermite-Lorentz $G$-homogeneous space for which the  isotropy algebra $\mathfrak{h}$  is isomorphic to ${\mathfrak{su}}(1,2)$  that is acting by its standard
representation on ${\mathfrak{g}}/{\mathfrak{h}}\equiv \mathbb{C}^{3}$. Let $\omega(., .)  = g(., J.)$ be the K\"ahler form 
of $(M, J, g)$. 

We will prove that $\omega $ is closed, and then $J$ is integrable. The idea is to observe that the $H$-orbits 
of $M$-points  are ``isomorphic''  to the $H$-orbits in $\C^3$. This allows one to understand the restriction of $\omega$ 
and $J$ to these $H$-orbits in $M$.

\subsection{ The $\mathfrak{su}(1, 2)$-action on $\C^3$.} Endow $\C^3$ with the usual Hermite-Lorentz form
$q= - |z_1 |^2 + |z_2|^2 +  | z_3|^2$. So,  $\U(1, 2)$ and 
$\SU(1, 2)$ both act transitively on each level of $q$ in $\C^3- \{0\}$.  A negative level
$\{ q = - \epsilon^2\}$ corresponds  to the $\SU(1, 2)$-orbit of some point  $(\epsilon, 0, 0) \in \C^3$. 
All of these ponits have the same compact isotropy $   \SU(2))$. So,  as  homogeneous spaces, all these orbits are isomorphic to  
$\SU(1, 2) / \SU(2)$.  Therefore, we get a  Hopf fibration $\SU(1, 2) / \SU(2) \to  \SU(1, 2) /\mathsf{S} ( \U(1) \times \U(2))$. 
The latter space (also denoted $\SU(1, 2) / \U(2)$) is nothing but the complex hyperbolic plane $\C \mathbb H^2$. 

\begin{Fact} \label{unique_2form}  {\it
Up to scaling, there exists exactly one $\SU(1, 2)$-invariant differential 2-form $\omega^0$
on $\SU(1, 2) / \SU(2)$. This is the restriction of the K\"ahler  form of $\C^3$, or equivalently, the pull back
of  the K\"ahler form on $\C \mathbb H^2$. In particular, such a form is closed.}
\end{Fact}

\begin{proof} The tangent space at $(\epsilon, 0, 0)$ of the level $\{q = -\epsilon^2\}$ 
is $\{ 0 \} \times \R \times \C^2$. The isotropy $\SU(2)$-action is the product of the 
trivial one on $\R$ by the usual one on $\C^2$. Let $\alpha$ be an anti-symmetric 2-form
on $\R \times \C^2$ that is invariant under the $\SU(2)$-action. Its kernel is $\SU(2)$-invariant and consequently it equals $\R$ or $\C^2$. 
It cannot be $\C^2$ unless $\alpha = 0$. Hence, it is $\R$. It follows that $\alpha $ projects to a $\SU(2)$-invariant form
on $\C^2$, which is known to be unique (up to scaling). \end{proof}

\subsection{The isotropy-action on $G/H$ via the exponential map.}
Let $%
p$ be any point of $M$, say $p$ is the base point $1H$, and consider the
exponential map $\exp _{p}$ defined on an open subset of $T_{p}M$ to $M$.
Then,  $\exp _{p}$ is equivariant with respect to the $H$-derivative action on 
$T_{p}M$ and the isometric $H$-action on $M$.

\begin{Fact} {\it  Let $\omega $ be the K\"ahler form of $M= G/H$. Let $u \in T_p M$ be a timelike 
vector (small enough, say, lies in a small neighbourhood of $ 0 \in T_pM$ where
 $\exp_p$ is a diffeomorphism onto  its image), and let $O_u = H. \exp_p u$ be the orbit of $\exp_p u$ under the isotropy $H$. Then $\omega$ is closed on $O_u$
}
\end{Fact}

\begin{proof} From the equivariance of $\exp_p$, it 
  follows in particular that $u$ and $\exp_p u$ have the same
isotropy in $H = \SU(1, 2)$. Therefore, as a $\SU(1, 2)$-homogeneous space, the orbit $O_u $
is identified to $\SU(1, 2) / \SU(2)$.  We conclude from 
Fact \ref{unique_2form} that the restriction of $\omega$ to $O_u$ is closed.\end{proof}

\subsection{Proof that $M$ is almost (pseudo)-K\"ahlerian.} 
Let $u$ be, as above, a small timelike vector in $T_pM$, and $\gamma_u(t) $
the  geodesic with $\dot{\gamma_u}(0) = u$. By Gauss lemma, the orthogonal ${\gamma_u}(t)^\perp$ 
coincides with the tangent space at $\gamma_u(t)$ of   the image by $\exp_p$ of the 
level $\{q = t^2 q(u)\}$ in $T_pM$ ($q$ is the Hermite-Lorentz form on $T_pM)$. Because
these levels in $T_pM$ correspond to $H$-orbits in $T_pM$, their images by $\exp_p$ correspond 
to the $H$-orbits in $M$. Therefore, the tangent space of the orbit $H. \gamma_u(t)$ at $\gamma_u(t)$
is $\dot{\gamma_u}^\perp(t)$.  It follows from the fact above that $d \omega $ vanishes on $\dot{ \gamma_u}^\perp(t)$.  
When $t \to 0$, we get that $d \omega $ vanishes on the real hyperplane $u^\perp \subset T_pM$. \\
Of course, this applies for any timelike vector $u \in T_pM$. It follows that $d\omega (u_1, u_2, u_3) = 0$, for 
any $(u_1, u_2, u_3 ) \in T_pM$ such that $(\R u_1 \oplus \R u_2 \oplus \R u_3)^\perp$ contains a timelike vector. 
But this subset of $T_pM \times T_p M \times T_p M$ is open and non-empty.  We conclude that $d \omega$ vanishes on $T_pM$, and that $\omega $ is closed on $M$ by homogeneity. 

\subsection{Integrability of the almost-complex structure.}
On the homogeneous space $\SU(1, 2)/ \SU(2)$, there is also   an invariant  CR-structure inherited from its 
realization as a real  hypersurface in $\C^3$. This is a $\SU(1, 2)$-invariant codimension one distribution $\mathcal C$ that is endowed with an almost complex structure $J_0$. As above, the distribution  $\mathcal C$ is  given at a point $(\epsilon, 0, 0)$  
by the factor $\C^2$ in the tangent space identified to  $  \R \times \C^2$.

Just as for the K\"ahler form, this invariant CR-structure on $\SU(1, 2) / \SU(2)$  is
unique. In fact, $\mathcal C$ is the unique $\SU(1, 2)$- invariant codimension one distribution, and the almost complex structure $J_0$ on it is unique.


Recall that, by definition of a CR structure (see for instance \cite{Jaco}, Chapter 1),  the  almost complex structure  $J_0$ on $\mathcal C$ is 
  integrable in the sense that, for any vector fields $X$ and $Y$ tangent to $\mathcal C$, 
$A= [J_0X, Y] + [X, J_0Y] $ and $B= [J_0X, J_0Y] - [X, Y]$ are tangent to $\mathcal C$ and $J_0 A = B$.

As above, 
consider an orbit $O_u= H. \exp_p u $, and remember that $J$  denotes the almost complex structure of $M$. The distribution 
$\mathcal D = T O \cap J (T O)$ is $\SU(1, 2)$-invariant and hence, by uniqueness, it  coincides with the 
CR structure $\mathcal C$. In particularm it satisfies the integrability condition 
 $0 =  \left[ JX,JY\right] -\left[ X,Y\right] -J\left[ JX,Y%
\right] -J\left[ X,JY\right] $, where $X$ and $Y$ are vector fields on $O$ tangent to $\mathcal D$.
By definition of the Nijenhuis tensor $N_J$ of $M$  (see Section 
\ref{computations}), this exactly means that $N_J(X, Y) = 0$ for $X$ and $Y$ tangent to $\mathcal D$.

By  considering a  limit  process as previously, we get that 
$N_J$ vanishes on $(\C u)^\perp$  for any $u \in T_pM$ timelike.  Again, as above, the non-empty space of elements $(v, w) \in (T_pM)^2$ such that $(\C v \oplus \C w)^\perp$ contains a timelike vector is open in $(T_pM)^2$. It follows that 
$N_J$ vanishes at $p$ (and hence vanishes everywhere), that is $J$ is integrable.

\subsection{End of the proof.}
In conclusion, $(M, J, g) $ is pseudo-K\"ahlerian, that is $J$ is integrable and the K\"ahler form of $g$ is closed. 
Observe now that $M$ has a  constant holomorphic sectional
curvature, since $\mathsf{SU}(1,2)$ acts transitively on vectors of the same
causal character in $\mathbb{C}^{3}$. 
Therefore, up to scaling, $M$ is locally isomorphic
to one of the universal K\"{a}hler-Lorentz spaces of constant holomorphic sectional curvature:   $\mathsf{Mink}_{3}(\mathbb{C})$ (i.e. $ \mathbb{C}^{3}$ endowed with the form $q$), $\mathsf{dS}_{3}(\mathbb{C})=\mathsf{SU%
}(1,3)/\mathsf{{U}(1,2)}$, or $\mathsf{AdS}_{3}(\mathbb{C})=\mathsf{SU}(1,3)/%
\mathsf{{U}(1,2)}$. \\ Exactly, as in \cite{BZ}, thanks to the high
homogeneity of $M$ (especially, its irreducible isotropy), we see that, up to
a central cover, $M$ equals one of these universal  spaces. In particular, the
infinitesimal isotropy is ${\mathfrak{u}}(1,2)$. Therefore, as we noticed at the beginning of the current section, the isotropy algebra cannot be equal to $\mathfrak{su}(1, 2)$. $\Box$


\section{Structure of the semi-simple part of ${\mathfrak{g}}$}

\label{semisimple}

Let $G/H$ be an almost Hermite-Lorentz homogeneous space of complex dimension 3,
with a $\mathbb{C}$-irreducible isotropy.

From paragraphs \ref{Overview} and \ref{su(1,2)}, we just need to
consider the case ${\mathfrak{h}}={\mathfrak{so}}(1,2)$. We have $\dim_{\mathbb{R}}
G=3\times 2+\dim_{\mathbb{R}} {\mathfrak{so}}(1,2)=9$. So our problem reduces to
finding a real Lie group $G$ having dimension 9 and containing a subgroup $H$ which is
locally isomorphic to $\mathsf{SL}(2,\mathbb{R})$ and such that the $H$-action
of ${\mathfrak{g}}/{\mathfrak{h}}$ is isomorphic to ${\mathfrak{hl}}$.

\medskip
In what follows, for convenience, we will use the symbol "$\oplus_S$" to denote a semi direct sum of Lie algebras.  

\medskip
Consider a Levi decomposition ${\mathfrak{g}} = {\mathfrak{s}} \oplus_S {%
\mathfrak{r}}$, where ${\mathfrak{s}}$ is semi-simple and ${\mathfrak{r}}$
is the radical of $\mathfrak{g}$.

\begin{lemma}
\label{semisimple_part} ${\mathfrak{s}}$ is isomorphic to either ${\mathfrak{sl}}(2,\mathbb{R})\oplus {\mathfrak{sl}}(2,%
\mathbb{C})$, or  $\oplus_{i=1}^{k}\mathfrak{sl}(2,\mathbb{R})$ with $k=1,2$, or $3$. 
\end{lemma}

\begin{proof}
${\mathfrak{s}}$ is a real semi-simple Lie algebra of $\dim \leq 9$. Since it contains ${\mathfrak{h}}$, ${\mathfrak{s}}$ contains a nonzero simple factor of noncompact type, say  ${\mathfrak{s}}_{1}$. 

We will show that ${\mathfrak{s}}_{1}$ is ${\mathfrak{sl}}(2,\mathbb{R})$, 
${\mathfrak{sl}}(2,\mathbb{C}), {\mathfrak{su}}(1,2)$, or ${\mathfrak{sl}}(3,%
\mathbb{R})$. Indeed, assume that ${\mathfrak{s}}_{1}$ has real rank 1. Then,
it is the Killing algebra of a real, complex, or quaternionic hyperbolic space, or a Cayley hyperbolic plane.
But only $\mathbb{R}H^{2}$, $\mathbb{R}H^{3}$ and $\mathbb{C}H^{3}$
have Killing algebra of dimension $\leq 9$. They are ${\mathfrak{sl}}(2,%
\mathbb{R}), {\mathfrak{so}}(1,3)\cong {\mathfrak{sl}}(2,\mathbb{C})$, and ${%
\mathfrak{su}}(1,2)$, having dimensions 3, 6, and 8, respectively.

If ${\mathfrak{s}}_1$ has rank $\geq 2$, then it contains ${\ \mathfrak{sl}%
}(3, \mathbb{R})$, or $\mathfrak{sp}(4, \mathbb{R})$ (see \cite{Mar},
Proposition I.1.6.2). They have dimensions 8 and 10, respectively.

Assume that ${\mathfrak{g}}$ contains ${\mathfrak{sl}}(3, \mathbb{R})$. Then 
${\mathfrak{sl}}(3, \mathbb{R})$ has codimension 1 in ${\mathfrak{g}}$. It follows that ${\mathfrak{g}}$
is not semi-simple, since codimension one subalgebras of semisimple Lie algebras come only
from ${\mathfrak{sl}}(2, \mathbb{R})$-factors. Therefore ${\mathfrak{g}}$
has a non-trivial radical which is isomorphic to $\mathbb{R}$ on which ${\mathfrak{%
sl}}(3, \mathbb{R})$ acts trivially, that is ${\mathfrak{g}} = {\mathfrak{sl%
}}(3, \mathbb{R}) \oplus \mathbb{R}$. The isotropy ${\mathfrak{h}}$ being
simple is thus contained in ${\mathfrak{sl}}(3, \mathbb{R})$. This
contradicts the fact that ${\mathfrak{h}}$ acts $\mathbb{C}$-irreducibly on $%
{\mathfrak{g}} / {\mathfrak{h}}$.

The same argument applies equally to ${\mathfrak{s}}_1= {\mathfrak{su}}(1,
2) $. Thus, we have proved that any factor of ${\mathfrak{s}}$ is isomorphic to ${%
\mathfrak{sl}}(2, \mathbb{R})$ or ${\mathfrak{sl}}(2, \mathbb{C}).$

\smallskip

If ${\mathfrak{s}}$ contains ${\mathfrak{sl}}(2,\mathbb{C})$, then it equals
either ${\mathfrak{sl}}(2,\mathbb{C})$ or ${\mathfrak{sl}}(2,\mathbb{C}%
)\oplus {\mathfrak{sl}}(2,\mathbb{R})$. In the first case, ${\mathfrak{g}}$
is a semi-direct sum ${\mathfrak{sl}}(2,\mathbb{C})\oplus_S {\mathfrak{r}}
$. But ${\mathfrak{sl}}(2,\mathbb{C})$ cannot act non-trivially on a Lie
algebra of real dimension 3. So ${\mathfrak{g}}$ is a direct sum. But, this
contradicts the fact that the isotropy is $\mathbb{C}$-irreducible. 
\end{proof}

\section{Identification of ${\mathfrak{g}}$}

\label{Group_G}

By Lemma \ref{semisimple_part}, ${\mathfrak{g}}$ is either semi-simple and isomorphic to ${\mathfrak{sl}}(2, \mathbb{R})\oplus {\mathfrak{sl}}(2, \mathbb{R})\oplus {\mathfrak{sl}}(2, \mathbb{R})$ or ${\mathfrak{sl}}(2, \mathbb{R}) \oplus {\mathfrak{sl}}(2, \mathbb{C})$, or 
${\mathfrak{g}} $ is a semi-direct sum $({{\mathfrak{sl}}(2, \mathbb{R}%
)\oplus {\mathfrak{sl}}(2, \mathbb{R})}) \oplus_S {\mathfrak{r}}$ with $\dim {\mathfrak{r}} = 3$, or ${\mathfrak{sl}}%
(2, \mathbb{R}) \oplus_S {\mathfrak{r}}$ with $\dim {\mathfrak{r}} = 6$.

Let us analyze each of these cases, detremine explicitely the radical ${\mathfrak{r}}$ and 
the semi-direct product, and see how we can embed the isotropy ${\mathfrak{h%
}}$.

\subsection{Case ${\mathfrak{s}} = {\mathfrak{sl}}(2, \mathbb{R})\oplus {\mathfrak{sl}}(2, \mathbb{R})\oplus {\mathfrak{sl}}(2, \mathbb{R})$.}
Of course, ${\mathfrak{g}} = {\mathfrak{s}}$ in this case. The isotropy ${%
\mathfrak{h}}$ embeds by $a \in {\mathfrak{h}} \to (\rho_1(a) , \rho_2(a) ,
\rho_3(a) ) \in {\mathfrak{g}}$. Each $\rho_i$ is either
trivial or an isomorphism. If one $\rho_i$ is trivial, then the ${\mathfrak{h%
}}$-representation on ${\mathfrak{s}} / {\mathfrak{h}}$ will have a trivial
space corresponding to the projection of the $i$-factor (of ${\mathfrak{s}} = {\mathfrak{sl}}(2, \mathbb{R})\oplus {\mathfrak{sl}}(2, \mathbb{R})\oplus {\mathfrak{sl}}(2, \mathbb{R})$), but ${\mathfrak{hl}}$ has no trivial space. Therefore
all the $\rho_i$'s are isomorphisms, and after conjugacy in each factor and
identification of ${\mathfrak{h}}$ with ${\mathfrak{sl}}(2, \mathbb{R})$, we
can assume that all the $\rho_i$ are the identity. 
Thus, ${\mathfrak{h}} \equiv {\mathfrak{sl}}(2, \mathbb{R}) $ embeds
diagonally in ${\mathfrak{s}} = {\mathfrak{sl}}(2, \mathbb{R})\oplus {\mathfrak{sl}}(2, \mathbb{R})\oplus {\mathfrak{sl}}(2, \mathbb{R})$.

So ${\mathfrak{h}}= \{ (a, a, a) \in {\mathfrak{sl}}(2, \mathbb{R})\}$. Any
element $(b, c, d)$ is equivalent mod ${\mathfrak{h}}$ to $(b-d, c-d, 0)$. The quotient space ${\mathfrak{s}}/{\mathfrak{h}}$ is in fact isomorphic to $%
{\mathfrak{sl}}(2, \mathbb{R}) \oplus {\mathfrak{sl}}(2, \mathbb{R}) \oplus
\{0\}$, and the ${\mathfrak{h}}$-action is the sum of two copies of the adjoint
representation of ${\mathfrak{sl}}(2, \mathbb{R})$, so it is ${\mathfrak{hl}}
$.

\subsection{Case ${\mathfrak{g}} = {\mathfrak{sl}}(2, \mathbb{R}) \oplus {%
\mathfrak{sl}}(2, \mathbb{C})$.}
Just as in the previous case, one shows that, up to an automorphism, ${%
\mathfrak{h}}$ embeds diagonally in ${\mathfrak{sl}}(2, \mathbb{R}) \oplus {%
\mathfrak{sl}}(2, \mathbb{R}) \subset {\mathfrak{sl}}(2, \mathbb{R}) \oplus {%
\mathfrak{sl}}(2, \mathbb{C})$.

\subsection{Case ${\mathfrak{s}} = {\mathfrak{sl}}(2, \mathbb{R})\oplus {\mathfrak{sl}}(2, \mathbb{R})$.}
The isotropy ${\mathfrak{h}}$ is simple in ${\mathfrak{g}}$, we can assume
it contained (up to conjugacy) in ${\mathfrak{s}}$. As in the previous case,
it must be, up to conjugacy, the diagonal of ${\mathfrak{sl}}(2, \mathbb{R}%
)^2$.

The radical ${\mathfrak{r}}$ has dimension 3. It projects injectively in ${%
\mathfrak{g}}/{\mathfrak{h}}$ (since ${\mathfrak{h}}\subset {\mathfrak{s}}$%
). Thus, ${\mathfrak{r}}$ is a factor of ${\mathfrak{hl}}$, i.e. it is
isomorphic to the $Ad$-action of ${\mathfrak{sl}}(2,\mathbb{R})$. Only $%
\mathbb{R}^{3}$, as a 3-dimensional solvable group, can admit such an
action. It follows that ${\mathfrak{r}}\cong \mathbb{R}^{3}$. The isotropy ${\mathfrak{h}}
$ is the diagonal of ${\mathfrak{sl}}(2,\mathbb{R})\oplus {\mathfrak{sl}}(2,%
\mathbb{R})$. The actions on $\mathbb{R}^{3}$ of the factors in the direct sum  ${\mathfrak{sl}}(2,\mathbb{R})\oplus {\mathfrak{sl}}(2,\mathbb{R})$ commute.  One can show this
happens in dimension 3 only when one factor acts trivially. In this case, we deduce that the other factor
acts like ${\mathfrak{h}}$ by $Ad$. Summarizing, ${\mathfrak{g}}={\mathfrak{%
sl}}(2,\mathbb{R})\oplus ({\mathfrak{sl}}(2,\mathbb{R})\oplus_S \mathbb{R}%
^{3})$, ${\mathfrak{sl}}(2,\mathbb{R})$ acts by $Ad$ on $\mathbb{R}^{3}$ and 
${\mathfrak{h}}$ embeds diagonally in ${\mathfrak{sl}}(2,\mathbb{R})\oplus {%
\mathfrak{sl}}(2,\mathbb{R})$.

\subsection{Case ${\mathfrak{s}} = {\mathfrak{sl}}(2, \mathbb{R})$.}
In this case, we have ${\mathfrak{g}} = {\mathfrak{sl}}(2, \mathbb{R}) \oplus_S r$, with $\dim {%
\mathfrak{r}} = 6$. As previously, we can identify ${\mathfrak{h}}$ with a subset
of the Levi factor ${\mathfrak{sl}}(2, \mathbb{R})$, and hence it equals this
factor. Therefore ${\mathfrak{r}} \sim {\mathfrak{g}} / {\mathfrak{h}}$ and
the ${\mathfrak{sl}}(2, \mathbb{R}) $-action on ${\mathfrak{r}}$ is ${%
\mathfrak{hl}}$.

Let us prove that ${\mathfrak{r}}$ is nilpotent. Indeed, if ${\mathfrak{n}}$
is the nilradical of ${\mathfrak{r}}$, then $\mathsf{{Aut}({\mathfrak{r}})}$
acts trivially on ${\mathfrak{r}}/{\mathfrak{n}}$ (see for instance \cite%
{Jac}), but this cannot happen if ${\mathfrak{r}}/{\mathfrak{n}}\neq 0$.

Now, if ${\mathfrak{r}}$ is not abelian, its center ${\mathfrak{c}}$ is ${\mathfrak{hl}}$%
-invariant. Hence ${\mathfrak{c}}$ is isomorphic to $\mathbb{R}^3$, on
which ${\mathfrak{sl}}(2, \mathbb{R})$ acts by $Ad$. The same thing happens for $[{%
\mathfrak{r}}, {\mathfrak{r}}]$, and consequently ${\mathfrak{c}} = [{\mathfrak{r}}%
, {\mathfrak{r}}]$. In particular, we deduce that ${\mathfrak{r}}$ is 2-step nilpotent.
Also, since ${\mathfrak{sl}}(2, \mathbb{R})$ acts by ${\mathfrak{hl}}$, we see that
there exists an invariant complement subspace ${\mathfrak{c}}^\prime$ of ${%
\mathfrak{c}}$ on which ${\mathfrak{sl}}(2, \mathbb{R})$ acts by $Ad$. The
bracket ${\mathfrak{c}}^\prime \times {\mathfrak{c}}^\prime \to {\mathfrak{c}%
}$ is ${\mathfrak{sl}}(2, \mathbb{R})$-equivariant. But the bracket ${%
\mathfrak{sl}}(2, \mathbb{R}) \times {\mathfrak{sl}}(2, \mathbb{R}) \to {%
\mathfrak{sl}}(2, \mathbb{R})$ is the unique such an equivariant skew-symmetric bilinear map.

By its correct definition, this means that ${\mathfrak{r}}$ is the Lie
algebra of $N_{{\mathfrak{sl}}_2}$.

\section{Identification to a left invariant metric on a Lie group $L$}

\label{Group_L}

In each of the four cases considered in the previous section, we will find a normal
subgroup $L$ which acts simply transitively on $G/H$ so that the latter space
can be identified to a left invariant almost complex Hermite-Lorentz
structure on $L$ having an isotropy group $H\subset \mathsf{{Aut}(L)}$ that acts
via ${\mathfrak{hl}}$. Of course, we can alternatively work with Lie algebras.

\subsection{Case ${\mathfrak{g}} = {\mathfrak{sl}}(2, \mathbb{R}) \oplus {%
\mathfrak{sl}}(2, \mathbb{R}) \oplus {\mathfrak{sl}}(2, \mathbb{R})$.}
Take $\mathfrak{l}={\mathfrak{sl}}(2,\mathbb{R})\oplus {\mathfrak{sl}}(2,%
\mathbb{R})\oplus \{0\}$. The infinitesimal isotropy ${\mathfrak{h}}$ being the diagonal, it follows that $%
\mathfrak{l}$ is transversal to ${\mathfrak{h}}$. More precisely, at the group
level, we have $G=LH$ and $L\cap H=1$.

\subsection{Case ${\mathfrak{g}} = {\mathfrak{sl}}(2, \mathbb{R}) \oplus {%
\mathfrak{sl}}(2, \mathbb{C})$.}
Take $\mathfrak{l}={\mathfrak{sl}}(2,\mathbb{C})$.

\subsection{Case ${\mathfrak{g}}={\mathfrak{sl}}(2,\mathbb{R})\oplus ({%
\mathfrak{sl}}(2,\mathbb{R})\oplus_S \mathbb{R}^{3})$.}
Here $\mathfrak{l}={\mathfrak{sl}}(2,\mathbb{R})\oplus_S \mathbb{R}^{3}.$

\subsection{Case ${\mathfrak{g}} = {\mathfrak{sl}}(2, \mathbb{R}) \oplus_S 
{\mathfrak{n}}_{{\mathfrak{sl}}_2}$.}
Here $\mathfrak{l}={\mathfrak{n}}_{{\mathfrak{sl}}_{2}}.$

\section{Integrability of the almost complex structure and Curvature Computations}\label{computations}

In what follows, we will make use of the following well known formulas
regarding the Lie group $L$ (see for instance \cite{bal}). We first note
that since $\mathfrak{l}$ is a subalgebra of ${\mathfrak{g,}}$ then we can
easily compute the Nijenhuis tensor of almost complex structure $J.$ For $%
X,Y\in \mathfrak{l},$ we can use the formula 
\begin{equation}
N_{J}\left( X,Y\right) =\left[ JX,JY\right] -\left[ X,Y\right] -J\left[ JX,Y%
\right] -J\left[ X,JY\right]   \label{1}
\end{equation}

To check that $M\cong L$ is K\"ahlerian or not, we compute the differential
of the K\"ahler form $\omega $ on $L$, which (we recall) is given by $\omega\left( X,Y\right) =\left\langle JX,Y\right\rangle$. We may use the formula 
\begin{equation}
d\omega \left( X,Y,Z\right) =\frac{1}{3}\left\{ -\omega \left( \left[ X,Y%
\right] ,Z\right) +\omega \left( \left[ X,Z\right] ,Y\right) -\omega \left( %
\left[ Y,Z\right] ,X\right) \right\} ,  \label{2}
\end{equation}%
for all $X,Y,Z\in \mathfrak{l}.$ Here, we have used the fact that on a the
Lie group $L$ the function $\left\langle X,Y\right\rangle $ is constant for
all $X,Y\in \mathfrak{l},$ where $\left\langle .,.\right\rangle $ is the (indefinite) inner product on $\mathfrak{l}$ induced from the left-invariant pseudo-Riemannian metric on $L$. 

To compute the Levi-Civita connection $\nabla$ on $L,$ we may use the formula%
\begin{equation}
\nabla _{E_{i}}E_{j}=\sum_{k=1}^{6}\frac{1}{2}\left( \alpha _{ijk}-\epsilon
_{i}\epsilon _{k}\alpha _{jki}+\epsilon _{j}\epsilon _{k}\alpha
_{kij}\right) E_{k},  \label{3}
\end{equation}%
where the $\alpha _{ijk}$ are the constant structures of $\mathfrak{l}$ with
respect to the orthonormal basis $\left\{ E_{1},\ldots ,E_{6}\right\} $ of $\mathfrak{l}$, and $\epsilon
_{i}=\left\langle E_{i},E_{i}\right\rangle $ (cf. \cite{albu} and \cite{mil}%
). We notice here that (\ref{3}) can be deduced from the Koszul formula, using the fact that each function $\left\langle E_{i},E_{j}\right\rangle $ is constant (The $ E_{i}$ are seen here as left-invariant vector fields on $L$).

We may compute $\nabla J$ according to the formula%
\begin{equation}
\left( \nabla _{X}J\right) Y=\nabla _{X}\left( JY\right) -J\left( \nabla
_{X}Y\right) .  \label{2'}
\end{equation}

\medskip

We will also make use of the basis $\left\{ X_{1},X_{2},X_{3}\right\} $ of ${%
\mathfrak{sl}}(2,\mathbb{R})$ given by%
\begin{equation}
X_{1}=\left( 
\begin{array}{cc}
1 & 0 \\ 
0 & -1%
\end{array}%
\right) ,~~~~~X_{2}=\left( 
\begin{array}{cc}
0 & 1 \\ 
1 & 0%
\end{array}%
\right) ,~~~~~X_{3}=\left( 
\begin{array}{cc}
0 & 1 \\ 
-1 & 0%
\end{array}%
\right) ,  \label{4}
\end{equation}%
for which the brackets are%
\begin{equation*}
\left[ X_{1},X_{2}\right] =2X_{3},~~~~~\left[ X_{1},X_{3}\right]
=2X_{2},~~~~~\left[ X_{2},X_{3}\right] =-2X_{1}.
\end{equation*}

This is an orthonormal basis for Lorentzian inner product%
\begin{equation}
\left\langle X,Y\right\rangle =\frac{1}{2}trace\left( XY\right) ,  \label{5}
\end{equation}%
for which $X_{3}$ is timelike. This coincides with the Killing form of ${%
\mathfrak{sl}}(2,\mathbb{R})$ which is of signature $\left( +,+,-\right) ,$
and it is well known that the corresponding left-invariant (in fact,
bi-invariant) Lorentz metric makes of $\mathsf{SL}(2,\mathbb{R})$ a Lorentz
space of constant curvature $-1$ (cf. \cite{nom}). It is nothing but the so-called anti-de-Sitter space $\mathsf{AdS}_{3}(\mathbb{R}).$

The Levi-Civita connection of $\mathsf{SL}(2,\mathbb{R})$ endowed with the
bi-invariant metric resulting from $\left\langle ,\right\rangle $ is simply
given by the formula 
\begin{equation}
\nabla _{X}Y=\frac{1}{2}\left[ X,Y\right].  \label{6}
\end{equation}

\subsection{Case ${\mathfrak{l}}={\mathfrak{sl}}(2,\mathbb{R})\oplus {%
\mathfrak{sl}}(2,\mathbb{R})$.}
With the notations above, let $\left\{ X_{1},X_{2},X_{3}\right\} $ be the
orthonormal basis of ${\mathfrak{sl}}(2,\mathbb{R})$ given by (\ref{4}), and
set 
\begin{equation*}
E_{i}=\left( X_{i},0\right) ~~\text{and }E_{i+3}=\left( 0,X_{i}\right)
,~~~~i=1,2,3.
\end{equation*}

This is an orthonormal basis of ${\mathfrak{l}}$ with $E_{3}$ and $E_{6}$
timelike vectors, for which the Lie brackets are given by 
\begin{equation*}
\left[ E_{1},E_{2}\right] =2E_{3},~~~\left[ E_{1},E_{3}\right] =2E_{2},~~~%
\left[ E_{2},E_{3}\right] =-2E_{1},
\end{equation*}%
\begin{equation*}
\left[ E_{4},E_{5}\right] =2E_{6},~~~\left[ E_{4},E_{6}\right] =2E_{6},~~~%
\left[ E_{5},E_{6}\right] =-2E_{4},
\end{equation*}%
and all others are zeros or can be deduced from the above ones by
antisymmetry.

The almost complex structure $J$ is given by%
\begin{eqnarray*}
&&J\left( E_{1}\right) =E_{4},~~~J\left( E_{2}\right) =E_{5},~\text{~}%
J\left( E_{3}\right) =E_{6}, \\
&&J\left( E_{4}\right) =-E_{1},~~~J\left( E_{5}\right) =-E_{2},~\text{~}%
J\left( E_{6}\right) =-E_{3}.
\end{eqnarray*}

In this case, the almost complex structure $J$ is not integrable. Using (\ref%
{1}), we have for example 
\begin{equation*}
N_{J}\left( E_{1},E_{2}\right) =2E_{6}-2E_{3}.
\end{equation*}

We also see that $M$ is not almost-K\"ahlerian. Using (\ref{2}), we have for
example 
\begin{equation*}
d\omega \left( E_{1},E_{2},E_{6}\right) =\dfrac{2}{3}.
\end{equation*}%
To compute the Levi-Civita connection on ${\mathfrak{l}},$ it suffices to
remember here that $L=\mathsf{SL}(2,\mathbb{R})\times \mathsf{SL}(2,\mathbb{R%
})$, where each factor $\mathsf{SL}(2,\mathbb{R})$ is endowed with the
bi-invariant metric resulting from $\left\langle ,\right\rangle $ for which
the Levi-Civita connection may be computed using (\ref{6}).

In other words, we have 
\begin{eqnarray*}
&\nabla _{E_{1}}E_{2}=&E_{3},~~~\nabla _{E_{1}}E_{3}=E_{2},~~~\nabla
_{E_{2}}E_{3}=-E_{1},~~~ \\
&\nabla _{E_{4}}E_{5}=&E_{6},~~~\nabla _{E_{4}}E_{6}=E_{5},~~~\nabla
_{E_{5}}E_{6}=-E_{4},~~~
\end{eqnarray*}%
and all the others are zeros or can be obtained by using the formula $\nabla
_{Y}X=\nabla _{X}Y-\left[ X,Y\right] .$

\begin{remark}
Since $L$ is a product of two Lorentz spaces of constant curvature, we
deduce that $M$ is a symmetric Pseudo-Riemannian manifold. However, $M$ is
not a locally symmetric Hermite-Lorentz space.
\end{remark}

\subsection{Case ${\mathfrak{l}}={\mathfrak{sl}}(2,\mathbb{C})$.}
In this case, we look at $\mathfrak{sl}(2,\mathbb{C})$ as the
complexification of $\mathfrak{sl}(2,\mathbb{R}),$ that is $\mathfrak{sl}(2,%
\mathbb{C})=\mathfrak{sl}(2,\mathbb{R})+\sqrt{-1}\mathfrak{sl}(2,\mathbb{R})$%
, with ${\mathfrak{sl}}(2,\mathbb{R})$ is endowed with the inner product (%
\ref{5}). Of course, this is not a direct sum, and note that the inner product  (%
\ref{5}) is extended in a sesquilinear way to $\mathfrak{sl}(2,\mathbb{C})$.
Consider the basis $\left\{ X_{1},X_{2},X_{3}\right\} $ of ${%
\mathfrak{sl}}(2,\mathbb{R})$ given by (\ref{4}), and set 
\begin{equation*}
E_{i}=\left( X_{i},0\right) ~~\text{and }E_{i+3}=\sqrt{-1}\left(
0,X_{i}\right) ,~~~~i=1,2,3.
\end{equation*}

This is an orthonormal basis of ${\mathfrak{l}}$ with $E_{3}$ and $E_{6}$
timelike vectors for which 
\begin{eqnarray*}
\left[ E_{1},E_{2}\right] &=&2E_{3},~~~\left[ E_{1},E_{3}\right] =2E_{2},~~~%
\left[ E_{2},E_{3}\right] =-2E_{1}, \\
\left[ E_{4},E_{5}\right] &=&-2E_{3},~~~\left[ E_{4},E_{6}\right]
=-2E_{2},~~~\left[ E_{5},E_{6}\right] =2E_{1}, \\
\left[ E_{1},E_{5}\right] &=&2E_{6},~~~\left[ E_{1},E_{6}\right] =2E_{5},~~~%
\left[ E_{2},E_{4}\right] =-2E_{6}, \\
\left[ E_{2},E_{6}\right] &=&-2E_{4},~~~\left[ E_{3},E_{4}\right]
=-2E_{5},~~~\left[ E_{3},E_{5}\right] =2E_{4},
\end{eqnarray*}%
and all other brackets are zeros or can be deduced by antisymmetry.

The almost complex structure $J$ is given by%
\begin{eqnarray*}
&&J\left( E_{1}\right) =E_{4},~~~J\left( E_{2}\right) =%
E_{5},~\text{~}J\left( E_{3}\right) =E_{6}, \\
&&J\left( E_{4}\right) =-E_{1},~~~J\left( E_{5}\right) =-E_{2},~\text{~}%
J\left( E_{6}\right) =-E_{3}.
\end{eqnarray*}

Of course, the almost complex structure $J$ is integrable in this case and
coincides with the natural complex structure of $\mathsf{SL}(2,\mathbb{C}).$%
\newline

The homogeneous space $M$ is not almost K\"{a}hlerian, since using (\ref{2}) we
have for example 
\begin{equation*}
d\omega \left( E_{1},E_{2},E_{6}\right) =-\dfrac{2}{3}.
\end{equation*}%
The Levi-Civita connection on $L$ may be computed by using (\ref{3}). We
obtain 
\begin{eqnarray*}
&\nabla _{E_{1}}E_{2}=&E_{3},~~~\nabla _{E_{1}}E_{3}=E_{2},~~~\nabla
_{E_{2}}E_{3}=-E_{1},~ \\
&\nabla _{E_{4}}E_{6}=&-E_{2},~~~\nabla _{E_{5}}E_{6}=-E_{2},~~~\nabla
_{E_{1}}E_{5}=3E_{6},~~~ \\
&\nabla _{E_{1}}E_{6}=&3E_{5},~~~\nabla _{E_{2}}E_{4}=-3E_{6},~~~\nabla
_{E_{2}}E_{6}=-3E_{4},~~~ \\
&\nabla _{E_{3}}E_{4}=&-3E_{5},~~~\nabla _{E_{3}}E_{5}=-3E_{4},~~~\nabla
_{E_{4}}E_{5}=-E_{3},
\end{eqnarray*}%
and all the other $\nabla _{E_{i}}E_{j}$ are zeros or can be obtained by
using the formula $\nabla _{Y}X=\nabla _{X}Y-\left[ X,Y\right] .$\newline

\begin{remark}
We can check that the Riemannian tensor $\mathcal{R}$ is not parallel by
computing for example 
\begin{equation*}
\nabla \mathcal{R}\left( E_{1},E_{2},E_{5},E_{6}\right) =6E_{2}.
\end{equation*}%
We deduce from this that $M$ is not locally symmetric as a pseudo-Riemannian
space (and a fortiori as an almost Hermite-Lorentz space).
\end{remark}

\subsection{Case $\mathfrak{l}={\mathfrak{sl}}(2,\mathbb{R})\oplus_S \mathbb{R%
}^{3}$.}
Now, the Hermite-Lorentz inner product $\left\langle ,\right\rangle ~$on $\mathfrak{l}={%
\mathfrak{sl}}(2,\mathbb{R})\oplus_S \mathbb{R}^{3}$ is obtained as follows. For $%
i=1,2,3,$ we set%
\begin{equation*}
E_{i}=\left( X_{i},0\right) ~~\text{and }E_{i+3}=\left( 0,e_{i}\right) ,~
\end{equation*}%
where $\left\{ e_{1},e_{2},e_{3}\right\} $ is the standard orthonormal
Lorentz basis of $\mathbb{R}^{3},$ with $e_{3}$ is timelike. By letting ${\mathfrak{sl}}%
(2,\mathbb{R})$ act on $\mathbb{R}^{3}$ by its adjoint representation, we obtain an
orthonormal basis of $\mathfrak{l}$ with $E_{3}$ and $E_{6}$ timelike
vectors and for which the Lie brackets are given by the following identities%
\begin{eqnarray*}
&&\left[ E_{1},E_{2}\right] =2E_{3},~~~\left[ E_{1},E_{3}\right] =2E_{2},~~~%
\left[ E_{2},E_{3}\right] =-2E_{1}, \\
&&\left[ E_{1},E_{5}\right] =2E_{6},~~~\left[ E_{1},E_{6}\right] =2E_{5},~~~%
\left[ E_{2},E_{4}\right] =-2E_{6}, \\
~~~ &&\left[ E_{2},E_{6}\right] =-2E_{4},~~~\left[ E_{3},E_{4}\right]
=-2E_{5},~~~\left[ E_{3},E_{5}\right] =2E_{4},
\end{eqnarray*}%
and all other brackets are zeros or can be deduced from the latter by
antisymmetry.

The almost complex structure $J$ is given by%
\begin{eqnarray*}
&&J\left( E_{1}\right) =E_{4},~~~J\left( E_{2}\right) =E_{5},~\text{~}%
J\left( E_{3}\right) =E_{6}, \\
&&J\left( E_{4}\right) =-E_{1},~~~J\left( E_{5}\right) =-E_{2},~\text{~}%
J\left( E_{6}\right) =-E_{3}.
\end{eqnarray*}

The complex structure $J$ is not integrable given that, using (\ref{1}), we
have for example%
\begin{equation*}
N_{J}\left( E_{1},E_{2}\right) =2E_{3}.
\end{equation*}

To deduce that the Hermite-Lorentz space $M=G/H$ is not almost-K\"ahler, we
may use (\ref{2}) to find for example%
\begin{equation*}
d\omega \left( E_{2},E_{3},E_{4}\right) =-\frac{2}{3}.
\end{equation*}

We may also use (\ref{2'}) to get for example that

\begin{equation*}
\left( \nabla _{E_{1}}J\right) E_{2}=E_{6}.
\end{equation*}

Regarding the connection, by using (\ref{3}), we obtain%
\begin{eqnarray*}
&\nabla _{E_{1}}E_{2}=&E_{3},~~~\nabla _{E_{1}}E_{3}=E_{2},~~~\nabla
_{E_{1}}E_{5}=2E_{6},~~~\nabla _{E_{1}}E_{6}=2E_{5},~~~ \\
&\nabla _{E_{2}}E_{3}=&-E_{1},~~~\nabla _{E_{2}}E_{4}=-2E_{6},~~~\nabla
_{E_{2}}E_{6}=-2E_{4},~~~ \\
&\nabla _{E_{3}}E_{4}=&-2E_{5},~~~\nabla _{E_{3}}E_{5}=2E_{4},
\end{eqnarray*}%
and all the others are zero or can be obtained by using the formula $\nabla
_{Y}X=\nabla _{X}Y-\left[ X,Y\right] .$

\begin{remark}
Of course, $M$ is not a symmetric Hermite-Lorentz space. However, it turns
out that $M$ is a symmetric pseudo-Riemannian manifold as we can easily show
that $\nabla \mathcal{R}\left( E_{i},E_{j},E_{k},E_{l}\right) =0,$ for all $%
i,j,k,l\in \left\{ 1,\ldots ,6\right\} ,$ where $\mathcal{R}$ is the
Riemannian tensor.
\end{remark}

\subsection{Case $\mathfrak{l}= \mathfrak{n}_{\mathfrak{sl}_{2}}$.}
With the notations of the previous case, let $\left\{
X_{1},X_{2},X_{3}\right\} $ be the orthonormal basis of ${\mathfrak{sl}}(2,%
\mathbb{R})$ given by (\ref{4}). Setting 
\begin{equation*}
E_{i}=\left( X_{i},0\right) ~~\text{and }E_{i+3}=\left( 0,X_{i}\right)
,~~~~i=1,2,3,
\end{equation*}%
we get an orthonormal basis of $\mathfrak{l}={\mathfrak{n}}_{{\mathfrak{sl}}%
_{2}}$ with $E_{3}$ and $E_{6}$ timelike vectors and for which the Lie
bracket are given by the following identities%
\begin{equation*}
\left[ E_{1},E_{2}\right] =2E_{6},~~~\left[ E_{1},E_{3}\right] =2E_{5},~~~%
\left[ E_{2},E_{3}\right] =-2E_{4},
\end{equation*}%
and all other brackets are zeros or can be deduced from the latter by
antisymmetry.

The almost complex structure $J$ is given by%
\begin{eqnarray*}
&&J\left( E_{1}\right) =E_{4},~~~J\left( E_{2}\right) =E_{5},~\text{~}%
J\left( E_{3}\right) =E_{6}, \\
&&J\left( E_{4}\right) =-E_{1},~~~J\left( E_{5}\right) =-E_{2},~\text{~}%
J\left( E_{6}\right) =-E_{3}.
\end{eqnarray*}

It is not integrable, as we can see for example from%
\begin{equation*}
N_{J}\left( E_{1},E_{2}\right) =-E_{6}.
\end{equation*}

It also turns out that the Hermite-Lorentz space $M=G/H$ is not
almost-K\"ahler. We have for example%
\begin{equation*}
d\omega \left( E_{1},E_{2},E_{3}\right) =-2.
\end{equation*}

The Levi-Civita connection on $L$ is computed using (\ref{3}). We obtain 
\begin{eqnarray*}
&\nabla _{E_{1}}E_{2}=&E_{6},~~~\nabla _{E_{1}}E_{3}=E_{5},~~~\nabla
_{E_{1}}E_{5}=E_{3},~~~\nabla _{E_{1}}E_{6}=E_{2},~~~ \\
&\nabla _{E_{2}}E_{3}=&-E_{4},~~~\nabla _{E_{2}}E_{4}=-E_{3},~~~\nabla
_{E_{2}}E_{6}=-E_{1},~~~ \\
&\nabla _{E_{3}}E_{4}=&-E_{2},~~~\nabla _{E_{3}}E_{5}=E_{1},
\end{eqnarray*}%
and all the others are zero or can be obtained by using the formula $\nabla
_{Y}X=\nabla _{X}Y-\left[ X,Y\right] .$

\begin{remark}
We can check that the Riemannian tensor $\mathcal{R}$ is not parallel by
computing for example 
\begin{equation*}
\nabla \mathcal{R}\left( E_{1},E_{2},E_{3},E_{6}\right) =E_{3}.
\end{equation*}%
We deduce from this that $M$ is not locally symmetric as a pseudo-Riemannian
space (and a fortiori as an almost Hermite-Lorentz space).
\end{remark}

\subsection{Full isometry group.}
\begin{Fact} {\it 
\label{Full_Group} In each of the four examples of spaces $G/H$ or
alternatively the Lie groups $L$ for which $\mathfrak{h} = \mathfrak{so}(1, 2)$ (described in 
Sections \ref{Group_G} and \ref{Group_L}), 
$G$ is exactly  the identity component of the full
isometry group   $\Iso(G/H)$ (to mean the automorphism group of the almost
complex Hermite-Lorentz structure). Equivalently, $H$ is the identity
component of the full isotropy group.

Furthermore, none of these spaces is (locally) symmetric (as a Hermite-Lorentz space).
}
\end{Fact}

\begin{proof}
If the identity component $\Iso^0(G/H)$ of the isometry group of $G/H$ is bigger than $G$, that is if 
$\dim\Iso\left(G/H\right)>\dim G$, then the isotropy algebra contains strictly 
$\mathfrak{so}(1, 2)$. By Fact
\ref{Isotropy_classification}, the isotropy algebra $\mathfrak{h}$ is then 
$\mathfrak{so}(1, 2) \oplus \mathfrak{so}(2)$, $\mathfrak{su}(1, 2)$ or $\mathfrak{u}(1, 2)$. 
By Section \ref{su(1,2)}, $\mathfrak{h}$ can not be $\mathfrak{su}(1, 2)$. So, if 
$\mathfrak{h}$ is bigger than $\mathfrak{so}(1, 2)$, then it equals $\mathfrak{so}(1, 2) \oplus \mathfrak{so}(2)$  or $\mathfrak{u}(1, 2)$.
As stated in \ref{symmetric_case}, the case $\mathfrak{h} = \mathfrak{so}(1, 2) \oplus \mathfrak{so}(2)$ implies
the space is   symmetric.
Indeed, 
the $\SO(2)$-action on $\C^3$ is that given by the homotheties  $Z \to  e^{ i t} Z$. In particular 
$Z \to e^{ i \pi } Z = -Z$, that is the isotropy contains an element inducing 
$- \mathsf{Id}$ on the tangent space, and therefore  $G/H$ is symmetric.

Observe now that a (locally) symmetric almost Hermite-Lorentz space is
necessarily K\"ahler-Lorentz, that is, its almost complex structure is
integrable and it is K\"ahlerian.  From the  previous computations in the present Section \ref{computations}, 
none of our four spaces satisfies all these conditions. Therefore, the isotropy algebra for all these spaces
must be  $\mathfrak{so}(1, 2)$ and $\Iso^0(G/H) = G$. 
\end{proof}

\section{Geodesic completeness, Proof of Theorem \ref{completeness_theorem}} \label{completeness}

Recall first that a pseudo-Riemannian symmetric space $M$ is always geodesically complete, for the simple reason that any global symmetry at any point $p$ of $M$ reverses geodesics through $p$. (An interval of $\R$
which is invariant under reflection about all its points is necessarily $\R$ itself).

According to \cite{BZ}, almost complex Hermite-Lorentz homogeneous spaces of dimension $\geq 4$  with $\mathbb{C}$-irreducible isotropy are symmetric.

On the other hand, the left invariant metric on $L=\mathsf{SL}(2, \mathbb{R}) \times \mathsf{SL}(2, \mathbb{R})$ is in fact bi-invariant as we have noticed in Section \ref{product}.

Similarly, as explained in Section \ref{hiden_symmetry},  the left invariant metric on $L = \mathsf{SL}(2,\mathbb{R})\ltimes \mathbb{R}^{3}$ is isometric to a left invariant metric on $\mathsf{SL}(2,\mathbb{R})\times {\mathfrak{sl}}(2,\mathbb{R})$ which is in fact nothing but the direct metric product $(\mathsf{SL}(2,\mathbb{R}), \kappa) \times ({\mathfrak{sl}}(2,\mathbb{R}), \kappa)$, where is the Killing form of ${\mathfrak{sl}}(2,\mathbb{R})$. So, this is also a symmetric space.

It follows that in all the above mentioned cases, we have geodesically complete spaces.  

The case $L= N_{{\mathfrak{sl}}_2}$ is so special here. Indeed, it turns out that  for    a two-step nilpotent Lie group, 
any   left invariant pseudo-Riemannian metric is   geodesically complete (see \cite{G}).

\bigskip

Finally, and in order to achieve the proof of Theorem \ref{completeness_theorem}, it remains to study completeness for the case  $L=\mathsf{SL}(2, \mathbb{C})$. In fact, we will do it in a more general context, namely, we will prove the completeness of a certain class of left invariant pseudo-Riemannian metrics on semi-simple Lie groups of which our considered metric on  $\mathsf{SL}(2, \mathbb{C})$ appears as a special case.

\smallskip

To do this, we need first to recall some general facts about the geodesic flow of left invariant pseudo-Riemannian metrics on Lie groups.

Let  $L$ be a Lie group with Lie algebra $\mathfrak{l}$. It is well known that  giving a left-invariant pseudo-Riemannian metric $\left\langle .,.\right\rangle $ on $L$ is equivalent to that of a nondegenerate quadratic form (that we also denote by $\left\langle .,.\right\rangle $) on  $\mathfrak{l}$.

Since $\left\langle X,Y\right\rangle $ is constant for any vectors $X,Y\in\mathfrak{l} $ (seen as left-invariant vector fields on $L$), it follows from the Koszul formula that the Levi-Civita connection of $\left\langle .,.\right\rangle $ is given by the formula (cf. \cite{cheg})
\begin{equation}
\nabla_XY=\frac{1}{2}\left\{\left[X,Y\right]-ad_{X}^{\ast }Y-ad_{Y}^{\ast }X\right\},  \label{7}
\end{equation}
where $ad_{X}^{\ast }$ is the adjoint of $ad_{X}$ with respect to inner product $\left\langle .,.\right\rangle $ on $\mathfrak{l}$.

Every curve $c\left( t\right) $ in $L$ gives rise
to the curve $\gamma \left( t\right) =\left( dL_{c\left( t\right) }\right)
^{-1}\left( \dot c\left( t\right) \right) $ in $\mathfrak{l}$, where for $g\in L,$ $L_{g}$ denotes the left-invariant translation by $g$ in $L$. We have

\begin{equation}
\left( dL_{c\left( t\right) }\right)^{-1}\left( \nabla_{ \dot c\left( t\right) }{ \dot c\left( t\right) } \right)=\dot \gamma \left( t\right)+\nabla_{\gamma \left( t\right)}\gamma \left( t\right) \label{7a}
\end{equation}

By using (\ref{7}), we deduce that $c(t)$ is a geodesic in $L$ if and only if $\dot \gamma \left( t\right)=ad_{\gamma \left( t\right)}^{\ast}\gamma \left( t\right)$. In other words, we have shown that the curves of $\mathfrak{l}$ associated to the geodesics of $L$ are solutions of the equation (see also \cite{alek})%
\begin{equation}
\dot x=ad_{x}^{\ast }x \label{8}
\end{equation}

In the case where $L$ is semi-simple with Lie algebra $\mathfrak{l}$ and Killing form $\kappa$, it turns out that giving a left-invariant pseudo-Riemannian metric $\left\langle .,.\right\rangle $ on $L$ is also equivalent to that of a self-adjoint isomorphism $A$ of  $\mathfrak{l}$ such that
\begin{equation}
\left\langle X,Y\right\rangle=\kappa\left(X,A\left(Y\right)\right) ,  \label{9}
\end{equation}
for all $X,Y\in\mathfrak{l} $.

By substituting (\ref{8}) into (\ref{9}) and using the bi-invariance of $\kappa$, we obtain
\begin{eqnarray*}
\kappa\left(A\left(\dot x\right),y\right)&=&\left\langle\dot x,y\right\rangle \\
&=&\left\langle ad_{x}^{\ast }x,y\right\rangle \\
&=&\left\langle \left[x,y\right],x\right\rangle \\
&=&\kappa\left(\left[x,y\right],A\left(x\right)\right) \\
&=&-\kappa\left(\left[x,A\left(x\right)\right],y\right) \\
&=&\kappa\left(\left[A\left(x\right),x\right],y\right),
\end{eqnarray*}
for all $y\in\mathfrak{l} $.

Since $y$ is arbitrary and $\kappa$ is nondegenerate, we deduce that equation (\ref{8}) translates to the following equation (cf. \cite{her})
\begin{equation}
 A\left(\dot x\right)=\left[A\left(x\right),x\right]  \label{10}
\end{equation}

\bigskip

We are now in the position to give a proof of Theorem \ref{completeness_theorem} in the case of semi-simple Lie groups of which $L=\mathsf{SL}(2, \mathbb{C})$ appears as special case.

For this, and as we have mentioned  in Section \ref{complexification}, 
recall that the construction of the Hermite-Lorentz metric on $\mathsf{SL}(2, \mathbb{C}) $ is a general fact.
Indeed, let $P$ be a semi-simple (real) Lie group with Lie algebra ${\mathfrak{p}}$, and let $L= P^\mathbb{C}$ be its complexification with Lie algebra $\mathfrak{l}= {\mathfrak{p}}^\mathbb{C} = \mathfrak{p} + \sqrt{-1} {\mathfrak{p}} \sim {%
\mathfrak{p}} \oplus {\mathfrak{p}}$. 
By endowing   $\mathfrak{l}$
with $\kappa^P \oplus \kappa^P$, where $\kappa^P$ is the Killing form of ${\mathfrak{p}}$, we obtain a left invariant pseudo-Hermitian metric $h$ on $L$ that is invariant under the adjoint action $Ad_P$. Denote by $h_e$ its restriction to the Lie algebra $\mathfrak{l}$. This can be expressed in terms of the Killing form $\kappa^L$ of $\mathfrak{l}$, seen as a real  semi-simple Lie algebra.
More precisely, we have
\begin{equation*}
h_e\left( .,.\right) =\kappa^L \left( .,A\left( .\right)\right),
\end{equation*}
where $A$ is the self-adjoint isomorphism of  $\mathfrak{l}= {\mathfrak{p}}^\mathbb{C} =  {\mathfrak{p}} \oplus {\mathfrak{p}}$ defined by $A\left(u , v\right)=\left(u,- v \right)$. 
Thus, the geodesic equation   (\ref{10}) becomes:
\begin{eqnarray*}
\dot u \left(t\right)&=&0, \\
\dot v \left(t\right)&=&2\left[v \left(t\right), u \left(t\right)\right]
\end{eqnarray*}

Therefore $u\left(t\right)$ is constant, and thus we get a linear equation for $v\left(t\right)$, and so solutions are defined on the whole $\R$.
 Hence, the left invariant pseudo-Hermitian metric $h$ on $L = P^\mathbb{C}$ is geodesically complete.  $\Box$
\smallskip

\noindent\small{{\bf Acknowledgement:} We would like to thank the referees for their valuable remarks and suggestions.}



\begin{thebibliography}{99}

\bibitem{albu} R. P. Albuquerque, 
\newblock {\em On Lie Groups with Left Invariant
semi-Riemannian Metric}, 
\newblock Public. Centro de Matematica da Universidade do
Minho,  (1998), 1--13.

\bibitem{alek} D. V. Alekseevskii and B. A. Putko, 
\newblock {\em Completeness of left-invariant metrics on Lie groups}, 
\newblock Funct. Anal. Appl., {\bf 21} (1987), no. 3, 233-234.

\bibitem{bal} 
W. Ballmann, 
\newblock {\em Lectures on K\"{a}hler manifolds}, 
\newblock ESI Lectures in Mathematics and Physics. European Mathematical Society (EMS), Zurich, (2006).

\bibitem{BZ} A. Ben-Ahmed and A. Zeghib, 
\newblock {\em On homogeneous Hermite-Lorentz spaces}, 
\newblock Asian J. Math., {\bf 20} (2016), no. 3, 531 - 552.

\bibitem{ber} M. Berger, 
\newblock {\em Les espaces sym\'etriques noncompacts}, 
\newblock Ann. Sci. Ecole Norm. Sup., {\bf 74} (1957), no. 3, 85--177.

\bibitem{cal} G. Calvaruso and M. C. L\'{e}pez, 
\newblock {\em Pseudo-Riemannian homogeneous structures}, 
\newblock Springer, (2019).

\bibitem{cheg} J. Cheeger and D.Ebin, 
\newblock {\em Comparison Theorems in Riemannian Geometry}, 
\newblock North-Holland, (1975).

\bibitem{CM1} V. Cort\'{e}s and B. Meinke, 
\newblock {\em Pseudo-Riemannian almost hypercomplex
homogeneous spaces with irreducible isotropy},  
\newblock J. Lie Theory,  {\bf 27} (2017), no.
4, 983--993.

\bibitem{CM2} V. Cort\'{e}s and B. Meinke, 
\newblock {\em Pseudo-Riemannian almost quaternionic homogeneous spaces with irreducible isotropy}, 
\newblock Geom. Dedicata, {\bf 196} (2018), 45--52.

\bibitem{Ghy} E. Ghys, 
\newblock {\em D\'{e}formations des structures complexes sur les espaces
homog\`{e}nes de $\mathsf{SL}(2,\mathbb{C})$}, 
\newblock  J. Reine Angew. Math., {\bf 468} (1995),
113--138.

\bibitem{G} M. Guediri, 
\newblock {\em Sur la compl\'{e}tude des pseudo-m\'{e}triques invariantes \`{a} gauche sur les groupes de Lie nilpotents}, 
\newblock Rend. Sem. Mat. Univ. Pol. Torino, {\bf 52} (1994),  371--376.

\bibitem{GL} M. Guediri and J. Lafontaine,  
\newblock {\em Sur la compl\'{e}tude des vari\'{e}t\'{e}s
pseudo-riemanniennes}, 
\newblock J. Geom. Phys., {\bf 15} (1995), no. 2, 150--158.

\bibitem{hel} S. Helgason, 
\newblock {\em Differential geometry, Lie groups, and symmetric spaces},
\newblock Academic Press, New York,  (1978).

\bibitem{her} R. Hermann, 
\newblock {\em Geodesics and classical mechanics on Lie groups}, 
\newblock J. Math. Phys.,  {\bf 13} (1972), no. 4, 460--464.

\bibitem{Jaco} H. Jacobowitz, 
\newblock {\em Introduction to CR structures},  
\newblock Mathematical Surveys and Monographs, Vol. {\bf 32} (1990). 

\bibitem{Jac} N. Jacobson, 
\newblock {\em Lie algebras},
\newblock Wiley, New York, (1962).

\bibitem{Mar} G. A. Margulis, 
\newblock {\em Discrete Subgroups of Semisimple Lie Groups},
\newblock Springer, New York, (1991).

\bibitem{Mars} J.E. Marsden, 
\newblock {\em On completeness of homogeneous pseudo-Riemannian manifolds}, I
\newblock ndiana Univ. Math. J.,  {\bf 22} (1973), 1065--1066.

\bibitem{mil} J. Milnor, 
\newblock {\em Curvatures of left invariant metrics on Lie Groups},
\newblock Advances in Mathematics,  {\bf 21} (1976), 293--329.

\bibitem{nom} K. Nomizu, 
\newblock {\em Left-invariant Lorentz metrics on Lie groups}, 
\newblock Osaka J. Math.,  {\bf 16} (1979), no. 1,  143--150.

\bibitem{one} B. O'Neill,  
\newblock {\em Semi-Riemannian Geometry}, 
\newblock Academic Press, (1983).

\bibitem{SL} A. Di Scala and T. Leistner, 
\newblock {\em Connected subgroups of $\mathsf{SO}%
(2,n)$ acting irreducibly on $\mathbb{R}^{2,n}$}, 
\newblock Israel J. Math., {\bf 182} (2011),  103--121.

\bibitem{Rag} M. S. Raghunathan, 
\newblock {\em Discrete Subgroups of Lie Groups},
\newblock Springer, New York,  (1972).

\bibitem{TO} S. Tachibana and M. Okumura, 
\newblock {\em On the almost-complex structure of tangent bundles of Riemannian spaces}, 
\newblock Tohoku Math. J.,  {\bf 14} (1962), no. 2,  156--161.

\bibitem{wolf1} J. A. Wolf, 
\newblock {\em On the Classification of Hermitian Symmetric Spaces}, 
\newblock J. Math. Mech., {\bf 13} (1964), no. 3, 489--495. 

\bibitem{wolf2} J. A. Wolf, 
\newblock {\em Spaces of constant curvature}, 
\newblock Publish or Perish, Boston,  (1974).

\bibitem{Zim} R. Zimmer, 
\newblock {\em Ergodic Theory and Semisimple Groups}, 
\newblock Ergodic theory and semisimple groups,  Boston, Birkh\"auser, (1984).


\end{thebibliography}
\end{document}